\documentclass[11pt]{amsart}
\usepackage[utf8x]{inputenc}
\usepackage{amsmath, amsthm, amssymb, amscd, comment, enumerate, textcomp, units,stmaryrd, wasysym}
\usepackage{cite}
\usepackage[all, arrow]{xy}\SelectTips{cm}{10}
\usepackage{lipsum}
\makeatletter
\def\@setcopyright{}
\def\serieslogo@{}
\makeatother
\begin{document}
\baselineskip=14.5pt

\bibliographystyle{plain}

 \numberwithin{equation}{section}
 \newtheorem{theorem}{Theorem}[section]
 \newtheorem{proposition}[theorem]{Proposition} 
 \newtheorem{lemma}[theorem]{Lemma}
 \newtheorem{corollary}[theorem]{Corollary}
 \newtheorem{conjecture}[theorem]{Conjecture}
 \newtheorem{definition}[theorem]{Definition}
 \newtheorem{question}[theorem]{Question}
 \newtheorem*{claim*}{Claim}
 \newtheorem{claim}{Claim}
 \newtheorem{example}[theorem]{Example}

 \newcommand{\eff}{f}
 \renewcommand{\thefootnote}{\fnsymbol{footnote}}
 \renewcommand{\O}{\mathcal{O}}
 \renewcommand{\labelenumi}{(\alph{enumi})}

\def\today{
  \ifcase\month\or
  January\or February\or March\or April\or May\or June\or
  July\or August\or September\or October\or November\or December\fi
  \space\number\day, \number\year}

\title[Congruences for $\protect\eff$ and $\omega$]{Congruences for Ramanujan's $f$ and $\omega$ functions via generalized Borcherds products}
\subjclass[2010]{11F03, 11F33}
\keywords{mock theta functions, Borcherds products, harmonic Maass forms}
\author[Berg]{Jen Berg}
\address{University of Texas at Austin, Department of Mathematics, Austin, TX 78712}
\email{jberg@math.utexas.edu}
\author[Castillo]{Abel Castillo}
\address{The University of Illinois at Chicago, Department of Mathematics, Statistics, and Computer Science, Chicago, IL 60607-70045}
\email{acasti8@uic.edu}
\author[Grizzard]{Robert Grizzard}
\address{University of Texas at Austin, Department of Mathematics, Austin, TX 78712}
\email{rgrizzard@math.utexas.edu}
\author[Kala]{V\'\i t\v ezslav Kala} 
\address{Purdue University, Department of Mathematics, West Lafayette, IN 47907}
\email{vita.kala@gmail.com}
\author[Moy]{Richard Moy}
\address{Northwestern University, Department of Mathematics, Evanston, IL 60208}
\email{ramoy88@math.northwestern.edu}
\author[Wang]{Chongli Wang}
\address{Columbia University, Department of Mathematics, New York, NY 10027}
\email{chongli@math.columbia.edu}

\begin{abstract}
Bruinier and Ono recently developed the theory of generalized Borcherds products, which uses coefficients of certain Maass forms as exponents in  infinite product expansions of meromorphic modular forms. Using this, one can use classical results on congruences of modular forms to obtain congruences for Maass forms. In this note we work out the example of Ramanujan's mock theta functions $f$ and $\omega$ in detail.
\end{abstract}
\maketitle


\section{Introduction}\label{intro}

The goal of this note is to prove some explicit congruences for the following two functions:
\begin{align*}
\omega(q) &= \sum_{n=0}^\infty a_\omega(n)q^n := \sum_{n=0}^\infty \frac{q^{2n(n+1)}}{(1-q)^2(1-q^3)^2\cdots(1-q^{2n+1})^2} \\&= 1 + 2q + 3q^2 + 4q^3 + 6q^4 + 8q^5 + 10q^6 + 14q^7 + 18q^8 + 22q^9 + 29q^{10} + 36q^{11} +  \cdots;
\\
f(q) &= \sum_{n=0}^\infty a_f(n)q^n := \sum_{n=0}^\infty \frac{q^{n^2}}{(1+q)^2(1+q^2)^2\cdots(1+q^n)^2} \\&= 1 + q  - q^2 + q^3 - q^6 + q^7 + q^9 - 2q^{10} + q^{11}  - q^{12} + 2q^{13} - 2q^{14} + 2q^{15} - q^{16} + \cdots,
\end{align*}
where $q = e^{2\pi i \tau}$ for $\tau$ in the complex upper half-plane $\mathbb{H}$.  The functions $f(q)$ and $\omega(q)$ are two of Ramanujan's famous ``mock theta functions'', which he described in his famous last letter to Hardy.  Thanks largely to the work of Zwegers (see \cite{zwegersthesis} and \cite{zwegers01}), we now understand the mock theta functions as the holomorphic parts of half-integral weight weak harmonic Maass forms.  See for example \cite{onovisions} for background on these functions.

In \cite{bo} Bruinier and Ono described generalized Borcherds products for weak harmonic Maass forms and proved that they are weight 0 meromorphic modular forms.  They specifically describe a family of such Borcherds products which is defined using a vector-valued form whose coefficients are defined in terms of the coefficients of $f(q)$ and $\omega(q)$.  This work indicates that one can prove a multitude of congruences for these mock theta functions by using the classical theory of congruences for modular forms; however, using this technique to obtain explicit congruences takes some effort.  Bruinier and Ono showed one example of this (only for $\omega$) in \cite{bo123}, and in the present paper we produce additional examples.

In order to state our results, we must introduce some auxiliary objects which will be defined in Section \ref{NB}.  We will define a vector-valued form with coefficients $c^+$, which will be given in terms of the coefficients $a_f$ and $a_\omega$ in Lemma \ref{cplus}.  For a fixed negative fundamental discriminant $\Delta$ and an integer $r$ such that $\Delta \equiv r^2 ~(\textup{mod}~24)$, the aforementioned Borcherds product studied by Bruinier and Ono has a logarithmic derivative which will have a \mbox{$q$-expansion} as follows:
\begin{align}\label{phi}
\Phi_{\Delta,r}(z, \widetilde{H}) =
\sum_{n=1}^\infty
  \sum_{d \mid n}
    c^+\left( \frac{| \Delta | d^2}{24}, \frac{rd}{12} \right)
    (-d)
    \sum_{s\ \mathrm{mod}\ \Delta}
      \left( \frac{\Delta}{s} \right)_2
      e \left(\frac{sn}{d \Delta}  \right)
q^n
\end{align}
where $\left( \frac{\bullet}{\bullet} \right)_2 $ is the quadratic residue symbol. For any prime $\ell$ inert or ramified in $\mathbb{Q}(\sqrt{\Delta})$ and any exponent $R$, the normalization $\Phi_{\Delta,r}^*=\sum_{n=1}^\infty b(n)q^n$ of this Fourier expansion will be the same modulo $\ell^R$ as that of a modular form on $\Gamma_0(6)$. Therefore it makes sense to discuss the action of Hecke operators $T_p$ for any prime $p\ne\ell$  on $\Phi_{\Delta,r}^*$, modulo $\ell^R$.  Our most general result is the following.

\begin{theorem}\label{mainthm}
Let $\Delta$ be a negative fundamental discriminant, $r$ an integer such that \mbox{$\Delta \equiv r^2 (\textup{mod}~24)$,} and let $\Phi_{\Delta,r}^*$ be the normalization of the form defined in (\ref{phi}).  Fix a prime $\ell$ inert or ramified in $\mathbb{Q}(\sqrt{\Delta})$ and a positive integer $R$.  Let $p$ be a prime such that 
\begin{equation}\label{eigenvector}
\Phi_{\Delta, r}^*~|~T_{p, k}\equiv b(p)\Phi_{\Delta, r}^* ~\left(\textup{mod}\ \ell^R\right),
\end{equation}
where $k$ is the corresponding weight as in Lemma \ref{lemma_Hasse} and $T_{p, k}$ is the appropriate Hecke operator. 

Write $c(n) = c^+\left( \frac{| \Delta | n^2}{24}, \frac{rn}{12} \right)$.  Then we have

\begin{equation}\label{cpluscong}
c(n) \equiv \frac{c(1)}{n} \prod_{p|n} \left(b(p^{v_p(n)})-b(p^{v_p(n)-1})\left(\frac{\Delta}{p}\right)_2 \right)~\left(\textup{mod}~\ell^R\right).
\end{equation}

Moreover, for a given $\Delta$, $r$, $\ell$, and $R$, the set of primes $p$ for which the above statement holds has positive arithmetic density.
\end{theorem}

For example, we have the following two corolarries.
\begin{corollary}\label{-8,4}
Let $\Delta=-8$, $r=4$.  Let $p\neq 2, 3$; $\ell$, $R$, and $k$ be as in Theorem \ref{mainthm}, such that $\Phi_{\Delta, r}^*~|~T_{p, k}\equiv 0 \left(\textup{mod}\ \ell^R\right)$.  Then for all $m \geq 1$ we have
\begin{align*}
a_\omega\left(\frac{2(p^{4m+2}-1)}{3}\right) &\equiv -\frac {1}{p^{2m+1}} (-p^{k-1})^m \left(\frac{-8} {p}\right)_2 \left(\frac{p} {3}\right)_2 ~\left(\textup{mod}\ \ell^R\right),~\textup{and}\\
a_\omega\left(\frac{2(p^{4m}-1)}{3}\right) &\equiv -\frac {1}{p^{2m}} (-p^{k-1})^m ~\left(\textup{mod}\ \ell^R\right).
\end{align*}
Explicitly, $\left(\frac{-8} {p}\right)_2=1$ when $p\equiv 1, 3\mod 8$ and $-1$ when $p\equiv 5,7\mod 8$;
$\left(\frac{p} {3}\right)_2=1$ when $p\equiv 1\mod 3$ and $-1$ when $p\equiv 2\mod 3$.

Moreover, for a fixed prime power $\ell^R$, the set of primes $p$ such that the above statement holds has positive arithmetic density.
\end{corollary}

\begin{corollary}\label{-23,1}
Let $\Delta=-23$, $r=1$.  Let $p\neq 2, 3, 23$; $\ell$, $R$, and $k$ be as in Theorem \ref{mainthm}, such that $\Phi_{\Delta, r}^*~|~T_{p, k}\equiv 0 \left(\textup{mod}\ \ell^R\right)$.  Then for all $m \geq 1$ we have
\begin{align*}
a_f\left(\frac{23p^{4m+2}+1}{24} \right) &\equiv -\frac {1}{p^{2m+1}} (-p^{k-1})^m \left(\frac{-23} {p}\right)_2 \left(\frac{p} {6}\right)_2 
~\left(\textup{mod} \ell^R\right),~\textup{and}\\
a_f\left(\frac{23p^{4m}+1}{24} \right) &\equiv -\frac {1}{p^{2m}} (-p^{k-1})^m ~\left(\textup{mod} \ell^R\right).
\end{align*}

Moreover, for a fixed prime power $\ell^R$, the set of primes $p$ such that the above statement holds has positive arithmetic density.
\end{corollary}

We will prove these and related results in Section \ref{proofs}.  We will describe the statements on the density of primes in Section \ref{sect_allmostall}.  In Section \ref{explicit}, we will describe how to prove explicit congruences that follow from our theorem, and give the following example in detail.
\begin{corollary}\label{exex}
We have
$$a_\omega\left(\frac{2(5^{2M}-1)}{3}\right) \equiv (-5)^{-m-\varepsilon} \textup{ (mod} ~23)$$ for each positive integer $M = 2m+\varepsilon$, $\varepsilon \in \{0,1\}$.
\end{corollary}

\section*{acknowledgments}
This paper arose from a project led by Ken Ono at the Arizona Winter School 2013.  The authors wish to thank Professor Ono and the organizers of the AWS for providing this opportunity and for the huge amount of support and help with the project. The authors also thank William Stein and \verb|SAGE| for help with our computations.


\section{Nuts and bolts}\label{NB}
\subsection{Modular forms modulo primes}
Let $A$ be the Hasse invariant. We recall that $A$ is a modified Eisenstein series that vanishes precisely at the supersingular locus,
and that for $\ell \geq 5$, $A$ can be lifted to the Eisenstein series $E_{\ell-1}$ (see \cite[Sections 2.0 and 2.1]{katz}).
Moreover, we know that \mbox{$E_{\ell-1} \equiv 1 \pmod \ell$ \cite[p. 6]{web}.}
We combine these facts in the following lemma.

\begin{lemma} \label{modularformF}
 For any rational prime $\ell \geq 5$ and any positive integer $R$, there exists a holomorphic modular form $F$, of weight $R(\ell-1)$
 such that $F \equiv 1 \pmod {\ell^R}$ and
 $F \pmod \ell$ vanishes precisely at $j(E) \pmod \ell$ for all elliptic curves $E$ supersingular over $\mathbb F_\ell$.
\end{lemma}

\subsection{``Almost all'' statements for congruences of modular forms}\label{sect_allmostall}

Following Ono \cite[p. 43]{web}, we make the following observations which result from the existence of Galois representations of modular forms and the Chebotarev Density Theorem.

\begin{lemma}Let $g(z)$ be an integral weight modular form in $M_k(\Gamma_0(6))$ whose Fourier coefficients are in $\O_K$, the ring of algebraic integers of a number field $K$. 
Let $\mathfrak{l}$ be an ideal of $\O_K$. Then a positive proportion of primes $p$ have the property that $$ g(z) \mid T_{p,k} \equiv 0 \pmod{\mathfrak{l}}.$$ \end{lemma}

\begin{lemma} Assuming the same notation as above, a positive proportion of primes $p$ have the property that
$$ g(z) \mid T_{p,k} \equiv 2 g(z) \pmod{\mathfrak{l}}.$$ \end{lemma}

In particular, for sufficient rational primes $\ell$ (see Lemma \ref{lemma_Hasse} below), we will have that the logarithmic derivative $\Phi_{\Delta,r}$ is congruent modulo $\ell$ to a modular form in $M_k(\Gamma_0(6))$, and is therefore an eigenfunction for the Hecke operator $T_{p,k}$ modulo $\ell$ with eigenvalue 0 or 2 for a positive proportion of primes $p$. Additionally, \cite[Theorem 2.65]{web} tells us that ``almost all'' of the coefficients of $\Phi_{\Delta, r}$ will be zero modulo $\ell$.  These results justify the statements on the density of primes $p$ satisfying the hypotheses of Theorem \ref{mainthm} and Corollaries \ref{-8,4} and \ref{-23,1}.

\subsection{Generalized Borcherds products}\label{gen}

We follow \cite[Section 8.2]{bo}, and write a weight 1/2 harmonic Maass for whose coefficients are related to those of $f$ and $\omega$.
Define the following cuspidal weight 3/2 theta functions
\begin{align*}
 g_1(\tau) &:= \sum_{n \in \mathbb Z} (-1)^n \left(n + \frac{1}{3} \right) e^{3 \pi i \left(n + \frac{1}{3} \right)^2 \tau },\\
 g_2(\tau) &:= \sum_{n \in \mathbb Z} - \left(n + \frac{1}{6} \right) e^{3 \pi i \left(n + \frac{1}{6} \right)^2 \tau },\text{  and}\\
 g_3(\tau) &:= \sum_{n \in \mathbb Z} \left(n + \frac{1}{3} \right) e^{3 \pi i \left(n + \frac{1}{3} \right)^2 \tau }.\\
\end{align*}
Let
$$
F(\tau) := \left(q^{-\frac{1}{24}} f(q),
                 2 q^{\frac{1}{3}}\omega \left( q^{\frac{1}{2}} \right),
                 2 q^{\frac{1}{3}}\omega \left(-q^{\frac{1}{2}} \right)
                 \right)^T
$$
and
$$
G(\tau) := \int_{-\overline \tau}^{i \infty} { \frac{\left( g_1(z),g_2(z),g_3(z) \right)^T}{\sqrt{-i(\tau - z)}}} dz;
$$
Zwegers showed that $H(\tau) := F(\tau) - G(\tau)$ is a vector-valued weight 1/2 harmonic Maass form (see \cite[Theorem 3.6]{zwegers01}; the term ``harmonic Maass form'' was defined later, cf. \cite{bo}).
Bruinier and Ono show in \cite[Lemma 8.1]{bo} that $H$ gives rise to $\widetilde H \in H_{1/2, \overline \rho_L}$;
the lemma below gives an explicit relation between the coefficients of $\widetilde H$ and those of $f$ and $\omega$.

\begin{lemma}\label{cplus} Let $c^+(m,h)$ denote the coefficients of the holomorphic part of $\widetilde H$.

 \begin{enumerate}
  \item
   \begin{align*}
  c^+ \left( m, \mathbb Z \right) =
  c^+ \left( m, \frac{3}{12} + \mathbb Z \right) =
  c^+ \left( m, \frac{6}{12} + \mathbb Z \right) =
  c^+ \left( m, \frac{9}{12} + \mathbb Z \right) =
  0
  \end{align*}
  \item
  \begin{align*}
  c^+ \left( m, \frac{1}{12} + \mathbb Z \right) &=
  -c^+ \left( m, \frac{5}{12} + \mathbb Z \right) =
  c^+ \left( m, \frac{7}{12} + \mathbb Z \right) =
  -c^+ \left( m, \frac{11}{12} + \mathbb Z \right) \\ &=
  a_f \left( m + \frac{1}{24}  \right) \\
  \end{align*}
  \item
  
  \begin{align*}
  c^+ \left( m, \frac{2}{12} + \mathbb Z \right) =
  -c^+ \left( m, \frac{12}{12} + \mathbb Z \right) =
  2  \left((-1)^{2 \left( m - \frac{1}{3} \right)} - 1   \right) a_\omega \left( 2 \left( m - \frac{1}{3} \right) \right)
  \end{align*}
  \item
  \begin{align*}
  c^+ \left( m, \frac{4}{12} + \mathbb Z \right) =
  -c^+ \left( m, \frac{8}{12} + \mathbb Z \right) =
  -2  \left((-1)^{2 \left( m - \frac{1}{3} \right)} + 1   \right) a_\omega \left( 2 \left( m - \frac{1}{3} \right) \right)
  \end{align*}
  \end{enumerate}
\end{lemma}
\begin{proof}
This is a simple computation following from the definition of $\widetilde H$ in \cite[Section 8.2]{bo}.
\end{proof}

One also observes that $\widetilde H$ satisfies the conditions
for the generalized twisted Borcherds lift $\Psi_{\Delta,r}(z, \widetilde{H})$ described in \cite[(8.10)]{bo} to be a meromorphic modular function of weight $0$ for $\Gamma_0(6)$, whose divisor is supported on CM points. We then use the following well-known lemma stated below without proof.

\begin{lemma}
 Let $g$ be a meromorphic modular function of weight $0$ for a congruence subgroup $\Gamma$ of $SL_2(\mathbb Z)$.
 Then, the logarithmic derivative of $g$, i.e.
 $$
 \frac{q\frac{d}{dq} g}{g},
 $$
 is a weight 2 meromorphic modular form for $\Gamma$ of weight $2$, with simple poles supported on the divisor of $g$, and no other poles.
\end{lemma}

We take the logarithmic derivative of $\Psi_{\Delta,r}(z, \widetilde{H})$, that is,
$$
\Phi_{\Delta,r}(z, \widetilde{H})
:=
\frac
{q \frac{d}{dq} \left[\Psi_{\Delta,r}(z,\widetilde{H}) \right]}
{\Psi_{\Delta,r}(z, \widetilde H)}
.
$$
By the lemma above, $\Phi_{\Delta,r}$ is a weight 2 meromorphic modular function for $\Gamma_0(6)$, with at most simple poles at CM points and no other poles.

The $q$-expansion of $\Phi_{\Delta,r}(z, \widetilde{H})$ is given by
\begin{equation}\label{phi_again}
\Phi_{\Delta,r}(z, \widetilde{H}) =
\sum_{n=1}^\infty
  \sum_{d \mid n}
    c^+\left( \frac{| \Delta | d^2}{24}, \frac{rd}{12} \right)
    (-d)
    \sum_{s\ \mathrm{mod}\ \Delta}
      \left( \frac{\Delta}{s} \right)_2
      e \left(\frac{sn}{d \Delta}  \right)
q^n,
\end{equation}
giving us an explicit relationship between the Fourier coefficients of $\Phi_{\Delta,r}(z, \widetilde{H})$
and the coefficients of the mock theta functions $f(q)$ and $\omega(q)$.

For simplicity of notation, let us denote $c(n):=c^+\left( \frac{| \Delta | n^2}{24}, \frac{rn}{12} \right)$ and $c:=c(1)$, suppressing the dependence on $r$ and $\Delta$ wherever it will not cause confusion. Also, let
$$\Phi^*=\Phi^*_{\Delta,r}(z, \widetilde{H}):=
\Phi_{\Delta,r}(z, \widetilde{H})/\left(-c \sum_{s\ \mathrm{mod}\ \Delta} \left( \frac{\Delta}{s} \right)_2 e \left(\frac{sn}{d \Delta}\right) \right)$$ be the normalized logarithmic derivative (i.e., so that its $q^1$-coefficient is 1).

To obtain a simpler expression for the coefficients of $\Phi^*$ we evaluate the ``Gauss sum'':

\begin{lemma}\label{gauss}
Let $a, \Delta\in\mathbb Z, a,\Delta\neq 0$. Let $\displaystyle G(a,\Delta)=\sum_{s\ \mathrm{mod}\ \Delta} \left(\frac\Delta s\right)_2 e\left(a\frac s \Delta\right)$.

\begin{enumerate} 
\item If $\gcd(a, \Delta)=1$, then $G(a, \Delta)=\left(\frac\Delta a\right)_2 G(1, \Delta)$.

\item  In particular, for $\Delta=-8$, we have
$$
G(a, -8)=
\left\{
  \begin{array} {ll}
    -\sqrt{-8}
    &
    \text{ if } a\equiv 1, 3 \pmod 8
    \\
    \sqrt{-8}
    &
    \text{ if } a\equiv 5, 7 \pmod 8
    \\
    0
    &
    \text{ if } a\equiv 0 \pmod 2.
  \end{array}
\right.
$$
\end{enumerate}
\end{lemma}

\begin{proof}
a) Let $a^{-1}$ be the inverse of $a$ $\pmod\Delta$. We have
\begin{align*}G(a, \Delta) & =  \sum_{s\ \mathrm{mod}\ \Delta} \left(\frac\Delta s\right)_2 e\left(\frac {as} \Delta\right) \\
& =   \sum_{s\ \mathrm{mod}\ \Delta} \left(\frac\Delta {a^{-1} s}\right)_2 e\left(\frac {s} \Delta\right) \\
& =  {\left(\frac{\Delta}{a^{-1}}\right)_2} \hspace{.5 em} \sum_{s\ \mathrm{mod}\ \Delta} \left(\frac\Delta {s}\right)_2 e\left(\frac {s} \Delta\right)
\end{align*}
by a simple substitution in the summation and the multiplicativity of the Kronecker symbol. Therfore, a) follows since $\left(\frac\Delta {a^{-1}}\right)_2=\left(\frac\Delta {a}\right)_2$.

The proof of b) is an explicit calculation of the values.
\end{proof}

\begin{proposition}\label{coeff}
Take $n$ coprime to $\Delta$.
Then the $q^n$ coefficient of the $q$-expansion of $\Phi^*_{\Delta, r}(z, \widetilde{H})$ is
$$b(n):=\frac{1}{c}\sum_{d|n} c(d)d\left(\frac{\Delta} {n/d}\right)_2.$$


Moreover,
$$c(n)=\frac{c}{n} \sum_{d|n} b(d)\mu(n/d)\left(\frac{\Delta} {n/d}\right)_2,$$

where $\mu(d)$ is the M\"{o}bius function.
\end{proposition}

\begin{proof}
The $q^1$-coefficient of $\Phi$ is $-c\cdot G(1, \Delta)$. Thus the formula for $b(n)$ follows immediately from Lemma \ref{gauss}.

Then note that we have the convolution formula $b=b_1*b_2$, where $b_1(n)=\frac{n}{c} c(n)$ and
$b_2(n)=\left(\frac{\Delta} {n}\right)_2$. Therefore, $b_2(n)$ is multiplicative, and so its Dirichlet (convolution) inverse is $\mu b_2$. Therefore $b_1=b*(\mu b_2)$, which gives exactly the formula for $c(n)$.
\end{proof}

In order to prove congruences in the next section, we need that $\Phi^*$ is congruent to modular forms modulo powers of primes:

\begin{lemma}\label{lemma_Hasse}
Suppose $\Phi_{\Delta,r}(z, \widetilde{H})$ (and thus also $\Phi^*$) has $B$ simple poles.
Let $\ell$ be a rational prime that remains inert or ramifies in $\mathbb Q (\sqrt \Delta)$.
Then there exists a holomorphic modular form $\Phi^o$ of weight $2+(\ell - 1)BR$ such that
$\Phi^* \equiv \Phi^o \pmod{\ell^R}$
\end{lemma}

\begin{proof}
Let $\alpha$ be a pole of $\Phi_{\Delta,r}(z, \widetilde{H})$. Then, $\alpha \pmod \ell$ corresponds to a zero of the holomorphic modular form $F$ defined in Lemma \ref{modularformF}.  We know that $\alpha$ will be a CM point defined over $\mathbb{Q}(\sqrt{\Delta})$, and that $\ell$ does not split in this field.
Suppose $\beta$ is a zero of $F$ with $\alpha \equiv \beta \pmod \ell$; letting $E_\alpha$ and $E_\beta$ be elliptic curves with $j$-invariant $\alpha$ and $\beta$ respectively, we can multiply $\Phi^*$ by
$$
 \left( F \cdot \frac{j(z)-j(\alpha)}{j(z) -j(\beta) } \right)^R
$$
to obtain a meromorphic modular form congruent to $\Phi^*$ modulo $\ell^R$ with one less simple pole. Observe that this increases the weight of the form by
$(\ell - 1) R$.
\end{proof}


\section{Proofs of congruences}\label{proofs}
Throughout this section, $\Delta$ will denote a fundamental discriminant and $r$ an integer such that $\Delta \equiv r^2 ~(\textup{mod}~24)$.

\begin{lemma}\label{hecke}
Let $g(q)=\sum_{n=0}^\infty a(n)q^n$ be a normalized (i.e., $a(1)=1$) modular form on $\Gamma_0(6)$ of weight $k$.
\begin{enumerate}
\item If $p\neq 2, 3$ is a prime such that $g\mid T_p=0$, then $a(p^{2m+1})=0$ and $a(p^{2m})=(-p^{k-1})^m$.

\item  Let $S$ be a set of primes $p\neq 2,3$ such that for all $p\in S$ we have $g~|~T_p=a(p)g$. Then $a(m)a(n)=a(mn)$ for all coprime $m, n$ divisible only by primes lying in $S$. Also, for $p\in S$ we have $\sum_{j=1}^\infty \frac{a(p^j)}{p^{js}}=\frac{1}{1-a(p)p^{-s}+p^{k-1-2s}}$.
\end{enumerate}
\end{lemma}

\begin{proof}
This is very standard, so we just provide a sketch:

The $q^n$-coefficient of $g|T_p$ is $a(pn)+p^{k-1}a(n/p)$.  Part (a) then follows by easy induction using this formula. Similarly, part (b) follows by using the formulas for coefficients of $g~|~T_p$.
\end{proof}

Proposition \ref{coeff} allows us to obtain congruences for the coefficients $c^+$ assuming we know the values of $b(n)$ modulo $\ell^R$ for a suitable prime power $\ell^R$. For example, such information can come from our logarithmic derivative's being an eigenfunction for a Hecke operator $T_p$ modulo $\ell^R$, or even an eigenfunction for all the Hecke operators (i.e. congruent modulo $\ell^R$ to a Hecke eigenform).

\begin{theorem}\label{theorem_cong}
Fix a prime $\ell$ which is inert or ramified in $\mathbb Q(\sqrt{\Delta})$ and $R\geq 1$.
Let $k$ be the corresponding weight as in Lemma \ref{lemma_Hasse} and $T_{p, k}$ the appropriate Hecke operator.  
Let $p\neq 2, 3, \ell$ be a prime such that $\Phi_{\Delta, r}^*\mid T_{p, k}\equiv 0 \pmod{\ell^R}$.  
 
If $p$ also does not divide $\Delta$, then:

$$c(p^{2m+1})\equiv -\frac {c}{p^{2m+1}} (-p^{k-1})^m \left(\frac{\Delta} {p}\right)_2 \pmod{\ell^R}$$

$$c(p^{2m})\equiv -\frac {c}{p^{2m}} (-p^{k-1})^m \pmod{\ell^R}$$

Moreover, (fixing $\Delta$, $r$, $\ell$, $R$), the set of primes $p$ such that $\Phi_{\Delta, r}^*\mid T_{p, k}\equiv 0 \pmod{\ell^R}$ has positive density.
\end{theorem}

\begin{proof}
By Proposition \ref{coeff} we have $$c(p^M)=\frac{c}{p^M} \sum_{N\leq M} b(p^N)\mu(p^{M-N})\left(\frac{\Delta} {p^{M-N}}\right)_2=
\frac{c}{p^M} \left(b(p^M)-b(p^{M-1})\left(\frac{\Delta} {p^{1}}\right)_2 \right),$$ because when $M-N\neq 0, 1$, $p^{M-N}$ is not square-free, and so $\mu(p^{M-N})=0$. The statement now immadiately follows from Lemma \ref{hecke}.
\end{proof}

\begin{proof}[Proof of Corollaries \ref{-8,4} and \ref{-23,1}]
Note that in this case $c=c^+\left(\frac{1}{1}, \frac{4}{12}\right)=-4 a_\omega(0)=-4$. 
By Lemma \ref{cplus} we have $c(p^{2m+1})=c^+\left( \frac{p^{4m+2}}{3}, \frac{4p^{2m+1}}{12} \right)=\pm 4a_\omega\left(\frac{2(p^{4m+2}-1)}{3}\right)$, where the sign $\pm$ is exactly $\left(\frac{p} {3}\right)_2$. The statement now follows immediately from \ref{theorem_cong}. 

The second part and Corollary \ref{-23,1} follow in the same way.
\end{proof}

Similarly we can obtain congruences for coefficients at $q^n$, where $n$ is divisible only by primes $p$ as in Theorem \ref{theorem_cong}.

\begin{proof}[Proof of Theorem \ref{mainthm}]
All the congruences in this proof are $\pmod{\ell^R}$, and so we omit writing this. Throughout the proof, let $n$ be a positive integer divisible only by primes in $S$. Note that by \mbox{Lemma \ref{hecke}} we have
$b(n)\equiv\prod_{p|n} b(p^{v_p(n)})$. Also the quadratic symbol $\left(\frac{\Delta} {\cdot}\right)_2$ is multiplicative and $\mu(d)$ is 0 if $d$ is not square-free.

Thus by Proposition \ref{coeff} we have
$$\frac{n}{c}c(n)=\sum_{d|n} b(n/d)\mu(d)\left(\frac{\Delta} {d}\right)_2 = \hspace{-1 em}\sum_{\begin{subarray}{c}
        d \mid n \\  d \text{ square-free}
      \end{subarray}} \hspace{-1 em} b(n/d)\mu(d)\left(\frac{\Delta} {d}\right)_2.$$

Writing $d=\prod_{p|n} p^{\varepsilon_p}$ with $\varepsilon_p=0, 1$, and using the multiplicativity of $b$, $\left(\frac{\Delta} {\cdot}\right)_2$, and $\mu$, we get
$$\frac{n}{c}c(n)\equiv \sum_{(\varepsilon_p)_p} \prod_{p|n} b(p^{v_p(n)-\varepsilon_p})(-1)^{\varepsilon_p}\left(\frac{\Delta} {p}\right)_2,$$
the sum being over all possible tuples $(\varepsilon_p)_p.$

But the last sum is just the product $\prod_{p|n} \left(b(p^{v_p(n)})-b(p^{v_p(n)-1})\left(\frac{\Delta}{p}\right)_2 \right)$, as we wanted to show.
\end{proof}


\section{Explicit computation}\label{explicit}
Our theorems, combined with the ``almost all'' theorems mentioned in Section \ref{sect_allmostall},  tell us that we can obtain infinitely many explicit congruences for the coefficients $a_f(n)$ and $a_\omega(n)$.  Writing any explicit congruence down, however, is a different matter.  If we want to prove that $\Phi^*_{\Delta,r}(z, \widetilde{H}) ~|~ T_p \equiv 0 \pmod \ell$ for certain primes $p$ and $\ell$, then we must compute enough coefficients of the $q-$expansion of $\Phi^*_{\Delta,r}(z, \widetilde{H})$ to be able to apply Sturm's Theorem (see \cite{sturm} and \cite[section 2.9]{web}); this useful theorem gives an upper bound on the number of coefficients of a modular form that can be zero modulo $\ell^R$ without all of them being zero.

There are two methods for computing these coefficients.  The first method uses \eqref{phi} directly.  These coefficients are given in terms of the coefficients of $f(q)$ and $\omega(q)$.  In practice, this is quite difficult, as we now describe.

Suppose we wanted to show that $\Phi^*_{\Delta,r} (q) ~|~ T_p \equiv 0 $ (mod $\ell$), by verifying this congruence for the first $N$ coefficients, where $N$ is determined using Sturm's Theorem.  This means we need to have at least $p \cdot N$ coefficients computed for $\Phi^*_{\Delta,r}(q)$.  Using Lemma \ref{cplus}, we can verify that this will require computing $\frac{|\Delta|(p \cdot N)^2+1}{24}$ coefficients of $f(q)$ and $\frac{|\Delta|(p \cdot N)^2-8}{12}$ coefficients of $\omega(q)$.  For example, if $\Delta = -8$, $N = 1000$, and $p = 5$, this would mean computing over 8 million coefficients for $\omega(q)$ and over 16 million for $f(q)$.

The second method is to find an explicit expression for $\Phi^*_{\Delta,r}$, as in \cite[Section 8.2]{bo}, in terms of well-known series.  The authors of \cite{bo} wrote an expression for $\Phi^*_{-8,4}$ as a rational function in Eisenstein series and Dedekind eta functions (see formulas on page 2175 of \cite{bo}).  Using this, one can compute as many coefficients for $\Phi^*_{-8,4}(q)$ as one can compute for the Eisenstein and eta products.  For the above example, this is faster by several orders of magnitude.

Using this second method, we were able to compute enough coefficients to prove some explicit examples of congruences for $\omega(q)$ using $\Phi^*_{-8,4}$.  We proved the following computationally.

\begin{example}\label{eigen} Let $\Phi^*_{-8,4} (q)$ be the normalized $q$-series defined in Section \ref{gen}.  We have
\begin{enumerate}
\item $\displaystyle \Phi^*_{-8,4} (q) ~|~ T_5 \equiv 0  \pmod {23}.$
\item Let $g$ be the normalized newform in $M_{10}(\Gamma_0(6))$. We have that $$\Phi_{-8,4}^* (q)\equiv g\pmod 5.$$
\end{enumerate}
\end{example}

We will only explicitly discuss part (a).  Part (b) follows in the same manner.  We know $\Phi_{-8,4}^* (q)$ has at most two poles from Bruinier and Ono's explicit formula \cite[p. 2174--2175]{bo}. We multiply $\Phi_{-8,4}^* (q)$ by two twisted Hasse invariants $\pmod 5$ to obtain a holomorphic modular form $\Phi^o\equiv\Phi^*_{-8,4}(q)\pmod{23}$ of weight $46$ on $\Gamma_0(6)$ by Lemma \ref{lemma_Hasse}. Using \verb|SAGE| (see\cite{sage}), we calculated that $\Phi^*(q) ~|~ T_5 \equiv 0  \pmod{23}$ to precision $q^{50}$. Sturm's Theorem gives us that if $h\in M_k(\Gamma)$, $\ell$ is a prime, and the coefficients of $q^n$ in the $q$-expansion of $h$ vanish modulo $\ell$ for all $n\le \frac{k\left[SL_2(\mathbb{Z})\colon \Gamma \right]}{24}$, then $h\equiv 0\pmod\ell$. In this case, the Sturm bound is $\frac{46\cdot\left[SL_2(\mathbb{Z})\colon \Gamma_0(6) \right]}{24}=23$ (a coincidence), and were able to verify that 
\begin{equation}\label{mod23}
\Phi^*_{-8,4} (q) ~|~ T_5\equiv\Phi^o(q) ~|~ T_5 \equiv 0  \pmod{23}
\end{equation} 
computationally, as described above.

Equation (\ref{mod23}), along with Corollary \ref{-8,4} gives us the explicit congruence
$$a_\omega\left(\frac{2(5^{2M}-1)}{3}\right) \equiv (-5)^{-m-\varepsilon} \textup{ (mod} ~23)$$ for each positive integer $M = 2m+\varepsilon$, $\varepsilon \in \{0,1\}$.  The first few cases of this congruence were verified by computer:

\vspace{7pt}
\renewcommand{\arraystretch}{2}
\begin{tabular}{| c | c | c | c |}
\hline 
 & $\frac{2(5^{2M}-1)}{3}$ & \shortstack{~\\$a_\omega\left(\frac{2(5^{2M}-1)}{3}\right)$} & congruence class mod 23 \\ \hline
$M=1$ & 16 & 101 & 9 \\ \hline
$M=2$ & 416 & 147019574355949 & 9 \\ \hline
\shortstack{$M=3$\\~\\~\\~\\~} & \shortstack{10416\\~\\~\\~\\~} & 
\shortstack{~\\~\\~\\
61055287817898262553386166573674890\\
71828343984275529702652915791815662\\
46045801}
& \shortstack{12\\~\\~\\~\\~}\\ \hline
\shortstack{$M=4$\\~\\~\\~\\~\\~\\~\\~\\~\\~\\~\\~\\~\\~\\~\\~\\~\\~\\~}& 
\shortstack{260416\\~\\~\\~\\~\\~\\~\\~\\~\\~\\~\\~\\~\\~\\~\\~\\~\\~\\~}&
\shortstack{~\\~\\~\\
47135043177557696996583864908485299\\
02877317903609618663661969973372815\\
28667228657386190726273269144570224\\
19421615118007256010585973847924915\\
81461923182898155305890742543478218\\
22773223303943277462010996573525810\\
14479578234053700509911978165482992\\
25987091881940041827267379817078651\\
91857035767135962999286799635887103\\
00979543276318001547509333358824720\\
65861382823046394902072007486105992\\
62211650315449}& 
\shortstack{12\\~\\~\\~\\~\\~\\~\\~\\~\\~\\~\\~\\~\\~\\~\\~\\~\\~\\~} \\ \hline
\end{tabular}
\vspace{7pt}

We can also produce an infinite family of such congruences modulo 5, because we verified that $\Phi^*_{-8,4}$ is congruent modulo 5 to a Hecke eigenform.


\bibliography{aws}{}

\end{document}